\documentclass[oneside,english]{amsart}
\usepackage[T1]{fontenc}
\usepackage[latin9]{inputenc}
\usepackage{amsthm}
\usepackage{amstext}
\usepackage{amssymb}
\usepackage{esint}

\usepackage{graphicx}
\usepackage{epsfig}
\usepackage[colorlinks=true, linkcolor=blue, citecolor=blue]{hyperref}

\makeatletter
\numberwithin{equation}{section}
\numberwithin{figure}{section}
\theoremstyle{plain}
\newtheorem{thm}{\protect\theoremname}[section]

\theoremstyle{plain}
\theoremstyle{definition}
\newtheorem{defn}[thm]{\protect\definitionname}
\theoremstyle{plain}
\newtheorem{lem}[thm]{\protect\lemmaname}
\newtheorem{cor}[thm]{\protect\corollaryname}
\theoremstyle{plain}

\theoremstyle{plain}
\newtheorem{exa}[thm]{\protect\examplename}
\theoremstyle{plain}
\makeatother

\usepackage{babel}
\providecommand{\definitionname}{Definition}
\providecommand{\lemmaname}{Lemma}
\providecommand{\theoremname}{Theorem}
\providecommand{\corollaryname}{Corollary}
\providecommand{\remarkname}{Remark}
\providecommand{\propositionname}{Proposition}
\providecommand{\examplename}{Example}

\DeclareMathOperator{\loc}{loc}
\DeclareMathOperator{\dv}{div}
\DeclareMathOperator{\ess}{ess}

\begin{document}

\title{Space Quasiconformal Mappings and Neumann Eigenvalues in Fractal Domains}

\author{V.~Gol'dshtein, R.~Hurri-Syrj\"anen,  and  A.~Ukhlov}

\address{Vladimir Gol'dshtein:
Department of Mathematics,
Ben-Gurion University of the Negev,
P. O. Box 653, Beer Sheva, 84105, Israel}
\email{vladimir@bgu.ac.il}

\address{Ritva Hurri-Syrj\"anen:
Department of Mathematics and Statistics,
Gustaf H\"allstr\"omin katu 2 b, 
FI-00014 University of Helsinki, Finland}
\email{ritva.hurri-syrjanen@helsinki.fi}

\address{Alexander Ukhlov: 
Department of Mathematics,
Ben-Gurion University of the Negev\,
P. O. Box 653, Beer Sheva, 84105, Israel}
\email{ukhlov@math.bgu.ac.il}

\thanks{R.~Hurri-Syrj\"anen whose  research was supported in part by a grant from
the Finnish Academy of Science and Letters, Vilho, Yrj\"o and Kalle V\"ais\"al\"a Foundation, is grateful for the hospitality given by the Department of Mathematics of the Ben-Gurion University of the Negev.
V.~Gol'dshtein's research was supported by the United States-Israel Binational Science Foundation (BSF Grant No. 2014055).	}
\date{\today}


\begin{abstract}
We study the variation of the Neumann eigenvalues of the $p$-Laplace operator under quasiconformal perturbations of space domains. This study allows to obtain lower estimates of the Neumann eigenvalues in fractal type domains. The suggested approach is based on the geometric theory of composition operators in connections with the quasiconformal mapping theory.

\end{abstract}
\maketitle

\maketitle
\markboth{\textsc{Space quasiconformal mappings and Neumann eigenvalues ...}}
{\textsc{V. Gol'dshtein,  R. Hurri-Syrj\"anen, and A. Ukhlov}}

\section{Introduction }

Let $\Omega$ be a bounded domain in the $n$-dimensional Euclidean space $\mathbb R^n\,,$ $n\geq 2$. We consider the Neumann eigenvalue problem for the $p$-Laplace operator, $p>1$,
\[
\begin{cases}
-\textrm{div}(|\nabla u|^{p-2}\nabla u)=\mu_p|u|^{p-2}u & \text{in $\Omega$}\\
\frac{\partial u}{\partial n}=0 & \text{on $\partial \Omega$}.
\end{cases} 
\]
This classical formulation is correct for bounded Lipschitz domains. The weak statement of this spectral problem is as follows: a function $u$ solves the previous problem, if and only if,
$u \in W^{1,p}(\Omega)$ and
$$
\int\limits _{\Omega} (|\nabla u(x)|^{p-2}\nabla u(x) \cdot \nabla v(x))\,dx =
\mu_p \int\limits _{\Omega} |u(x)|^{p-2} u(x) v(x)\,dx
$$ 
for all $v \in W^{1,p}(\Omega)$ and is correct for any bounded domains $\Omega$ in $\mathbb R^n$. 

The spectral stability estimates for elliptic operators were discussed, for example, in \cite{BL06,BL08,BLL08,CH53,D93,H05,K66}.
The main result of the article gives estimates of the variation of the first non-trivial Neumann eigenvalue $\mu_p$ under quasiconformal perturbations of domains: 

\vskip 0.5cm

{\bf Theorem A.} {\it  
Let a bounded domain $\Omega$ in $\mathbb R^n$ be a $(p,p)$-Sobolev-Poincar\'e domain, $1<p<\infty $. \color{black}Assume that there exists a Lipschitz $K$-quasiconformal homeomorphism $\varphi: \Omega\to\widetilde{\Omega}$ of a domain $\Omega$ onto a bounded domain $\widetilde{\Omega}$ such that 
$$
Q_{p}(\Omega):=\ess\sup\limits_{x\in\Omega}|D\varphi(x)|^{\frac{p-n}{p}}<\infty.
$$
Then 
$$
\frac{1}{\mu_p(\widetilde{\Omega})}\leq K
Q^p_{p}(\Omega)\||D\varphi|^n\mid L_{\infty}(\Omega)\|
\cdot \frac{1}{\mu_p(\Omega)}.
$$
}
\vskip 0.5cm

The lower estimates of the first non-trivial Neumann eigenvalues $\mu_p(\Omega)$ in basic domains $\Omega$ which can be union of convex domains will be considered in Section 3. Note, that if $\mu_p(\Omega)$ is calculated, then Theorem A gives lower estimates of $\mu_p(\widetilde{\Omega})$ in a large class of (non)convex domains in the terms of quasiconformal geometry of $\widetilde{\Omega}$.
The lower estimates of first non-trivial Neumann eigenvalues for convex domains in terms of Euclidean diameters and isoperimetric inequalities were intensively studied in the last decades  (see, for example, \cite{BCT15,ENT,FNT,PW}).

\vskip 0.5cm

As an example we consider the non-convex star-shaped domain $\Omega_{\delta}=\Omega_1\cup\Omega_2$, $\delta>0$ given, 
$\alpha=\delta (\sqrt{3}-1)/2$,
where 
\begin{equation*}
\Omega_1=\bigl\{(x',x_n)\in R^n\,: \max\{\vert x'\vert -\delta\,,-\alpha\} < x_n<\alpha\bigr\}
\end{equation*}
and
\begin{equation*}
\Omega_2=\bigl\{(x',x_n)\in R^n\,:  -\alpha< x_n<\min\{\delta -\vert x'\vert\,, \alpha\}\bigr\}\,.
\end{equation*}

Let $n=3$. 
Then, $\Omega_{\delta}=\Omega_1\cup\Omega_2$ is a $(\delta (\sqrt{3}-1)/2\,, \delta\sqrt{2}) $-John domain 
and there exists a $K$-quasiconformal mapping $\varphi:\mathbb R^3\to \mathbb R^3$ such that $\varphi(B^3(0,1))=\Omega_{\delta}$. By Theorem A (Example C) for $p>3$
$$
\mu_p(\Omega_{\delta})\geq 
\frac{\sqrt{2(4-\sqrt{6}-\sqrt{2})}}
{\delta^2(4+\sqrt{6}+\sqrt{2})^{1/4}
(\sqrt{4+\sqrt{6}-\sqrt{2}}+\sqrt{4-\sqrt{6}-\sqrt{2}})^2}\mu_p(B^3(0,1))\,.
$$

\vskip 0.5cm

The proof of Theorem A is based on estimates of constants in the Sobolev-Poincar\'e inequalities.

Let $1\leq r,p\leq \infty$. A bounded domain $\Omega$ in $\mathbb R^n$ is called a $(r,p)$-Sobolev-Poincar\'e domain, if for any function $f\in L^1_p(\Omega)$, the $(r,p)$-Sobolev-Poincar\'e inequality
$$
\inf\limits_{c\in\mathbb R}\|f-c\mid L_r(\Omega)\|\leq B_{r,p}(\Omega )\|\nabla f\mid L_p(\Omega)\|
$$
holds.

It is well known  that the constant $B_{r,p}(\Omega)$ depends on the geometry of $\Omega$, see, for example, \cite{M}.
We prove
\vskip 0.3cm

{\bf Theorem B.}
{\it Let a bounded domain $\Omega$ in $\mathbb R^n$ be a $(r,q)$-Sobolev-Poincar\'e domain, $1<q\leq r<\infty$. Suppose that  there exists a $K$-quasiconformal homeomorphism $\varphi: \Omega\to\widetilde{\Omega}$ of a domain $\Omega$ onto a bounded domain $\widetilde{\Omega}$, so that $\varphi$ 
 belongs to the Sobolev space $L^1_{\alpha}(\Omega)$ for some $\alpha>n$. Suppose additionally that
$$
Q_{p,q}(\Omega):=\left(\int\limits_{\Omega}|D\varphi|^{\frac{(p-n)q}{p-q}}~dx\right)^{\frac{p-q}{pq}}<\infty.
$$
for some $p\in[q,r)$.
Then for $1\leq s=\frac{\alpha-n}{\alpha}r$ in the domain $\widetilde{\Omega}$ the $(s,p)$-Sobolev-Poincar\'e inequality
holds and 
$$
B_{s,p}(\widetilde{\Omega})\leq K^{\frac{1}{p}}
\min\limits_{1\leq q<p}\left( Q_{p,q}(\Omega)\|D\varphi\mid L_{\alpha}(\Omega)\|^{\frac{n}{s}}\right)
\cdot B_{r,q}(\Omega),
$$
where $B_{r,q}(\Omega)$ is the best constant in the $(r,q)$-Sobolev-Poincar\'e inequality in the domain $\Omega$.
}

\vskip 0.5cm

The suggested method of investigation is based on the geometric theory of composition
operators \cite{U1,VU1} and its applications to the Sobolev
type embedding theorems \cite{GGu,GU}. 

In the recent works \cite{BGU1,BGU2,GPU17,GU16,GU2016} the spectral
stability problem and lower estimates of Neumann eigenvalues in planar domains were considered. In \cite{GU17} spectral estimates in space domains using the theory of weak $p$-quasiconformal mappings were obtained.

\section{Quasiconformal Composition Operators and Neumann Eigenvalues}

\subsection{Notation}

For any domain $\Omega$ in  $\mathbb{R}^{n}$ and any $1\leq p<\infty$
we consider the Lebesgue space of measurable  functions
 with the finite norm
\[
\|f\mid{L_{p}(\Omega)}\|:=\left(\int\limits _{\Omega}|f(x)|^{p}~dx\right)^{1/p}<\infty.
\]
 
The space
$L_{\infty}(\Omega )$  is the space of
essentially
bounded Lebesgue measurable functions with the finite norm 
\[
\|f\mid{L_{\infty}(\Omega)}\|:=\inf\bigl\{ b: \vert f(x)\vert\le b \mbox{ for almost every } x\in\Omega\bigr\}           <\infty.
\]

We define the Sobolev space $W^{1,p}(\Omega)$, $1\leq p<\infty$, 
as a Banach space of  weakly differentiable functions
$f:\Omega\to\mathbb{R}$ equipped with the following norm: 
\[
\|f\mid W^{1,p}(\Omega)\|:=\biggr(\int\limits _{\Omega}|f(x)|^{p}\, dx\biggr)^{\frac{1}{p}}+
\biggr(\int\limits _{\Omega}|\nabla f(x)|^{p}\, dx\biggr)^{\frac{1}{p}}.
\]

We define also the homogeneous seminormed space $L^1_p(\Omega)$
of  weakly differentiable functions $f:\Omega\to\mathbb{R}$ equipped
with the following seminorm: 
\[
\|f\mid L^1_p(\Omega)\|:=\biggr(\int\limits _{\Omega}|\nabla f(x)|^{p}\, dx\biggr)^{\frac{1}{p}}.
\]

We recall that any element of $L^{1}_{p}(\Omega )$ is in $L_{p,\loc}(\Omega)$, that is, the space of  functions  which are locally integrable to the power $p$ in $\Omega$, \cite{M}.

A mapping $\varphi:\Omega\to\mathbb{R}^{n}$ is weakly differentiable on $\Omega$, if its coordinate functions have weak derivatives on $\Omega$. Hence the formal Jacobi
matrix $D\varphi(x)$ and its determinant (Jacobian) $J(x,\varphi)$
are well defined at almost all points $x\in\Omega$. The norm $|D\varphi(x)|$
of the matrix $D\varphi(x)$ is the norm of the corresponding linear
operator. We will use the same notation for this matrix and the corresponding
linear operator.

We recall that a mapping $\varphi: \Omega\to\widetilde{\Omega}$ is called $K$-quasiconformal if $\varphi\in W^{1,n}_{\loc}(\Omega)$ and there exists a constant $K<\infty$ such that
$$
|D\varphi (x)|^n\leq K |J(x,\varphi)|\,\,\,\text{for almost all}\,\,\,x\in \Omega.
$$

A mapping $\varphi:\Omega\to\mathbb{R}^{n}$ possesses the Luzin $N$-property if an image of any set of measure zero has measure zero. Note that any Lipschitz mapping possesses the
Luzin $N$-property.

\subsection{Composition Operators on Lebesgue Spaces}

The following theorem  about composition operators on Lebesgue spaces is well known, we refer to  \cite{VU1}.

\begin{thm}
\label{thm:LpLq} Let a
homeomorphism $\varphi:\Omega\to\widetilde{\Omega}$ between two domains $\Omega$ and $\widetilde{\Omega}$
be weakly differentiable. Then the
composition operator 
\[
\varphi^{\ast}:L_{r}(\widetilde{\Omega})\to L_{s}(\Omega),\,\,\, 1\leq s\leq r<\infty,
\]
 is bounded, if and only if,  $\varphi^{-1}$ possesses the Luzin $N$-property
and 
\begin{gather}
\biggl(\int\limits _{\widetilde{\Omega}}\left|J(y,\varphi^{-1})\right|^{\frac{r}{r-s}}~dy\biggl)^{\frac{r-s}{rs}}=K<\infty, \,\,\,
\mbox{ for } 
1\leq s<r<\infty,\nonumber\\
\left|J(y,\varphi^{-1})\right|^{\frac{1}{s}}=K<\infty, \,\,\,
\mbox{ for } 
1\leq s=r<\infty.
\nonumber
\end{gather}
 The norm of the composition operator  is $\|\varphi^{\ast}\|=K$.
\end{thm}

\subsection{Composition Operators on   Sobolev Spaces}

By the standard definition functions of $L^1_p(\Omega)$ are defined only up to a set of measure zero, but they can be redefined quasi-everywhere i.~e. up to a set of $p$-capacity zero. Indeed, every function $u\in L^1_p(\Omega)$ has a unique quasicontinuous representation $\tilde{u}\in L^1_p(\Omega)$. A function $\tilde{u}$ is termed quasicontinuous if for any $\varepsilon >0$ there is an open  set $U_{\varepsilon}$ such that the $p$-capacity of $U_{\varepsilon}$ is less then $\varepsilon$ and on the set $\Omega\setminus U_{\varepsilon}$ the function  $\tilde{u}$ is continuous (see, for example \cite{HKM,M}).

Let $\Omega$ and $\widetilde{\Omega}$ be domains in $\mathbb R^n$. We say that
a homeomorphism $\varphi:\Omega\to\widetilde{\Omega}$ induces a bounded composition
operator 
\[
\varphi^{\ast}:L^1_p(\widetilde{\Omega})\to L^1_q(\Omega),\,\,\,1\leq q\leq p\leq\infty,
\]
by the composition rule $\varphi^{\ast}(f)=f\circ\varphi$, if for
any function $f\in L^1_p(\widetilde{\Omega})$,  the composition $\varphi^{\ast}(f)\in L^1_q(\Omega)$
is defined quasi-everywhere in $\Omega$ and there exists a constant $K_{p,q}(\varphi;\Omega)<\infty$ such that 
\[
\|\varphi^{\ast}(f)\mid L^1_q(\Omega)\|\leq K_{p,q}(\varphi;\Omega)\|f\mid L^1_p(\widetilde{\Omega})\|.
\]

The main result of \cite{U1} gives the analytic description of composition
operators on Sobolev spaces (we also refer to  \cite{VU1}) and asserts that

\begin{thm}
\label{CompTh} \cite{U1} A homeomorphism $\varphi:\Omega\to\widetilde{\Omega}$
between two domains $\Omega$ and $\widetilde{\Omega}$ induces a bounded composition
operator 
\[
\varphi^{\ast}:L^1_p(\widetilde{\Omega})\to L^1_q(\Omega),\,\,\,1\leq q<p<\infty,
\]
 if and only if, $\varphi\in W_{\loc}^{1,1}(\Omega)$, $\varphi$ has finite distortion,
and 
\[
K_{p,q}(\varphi;\Omega)=\biggl(\int\limits _{\Omega}\biggl(\frac{|D\varphi(x)|^{p}}{|J(x,\varphi|}\biggr)^{\frac{q}{p-q}}~dx\biggr)^{\frac{p-q}{pq}}<\infty.
\]
\end{thm}


We prove the following property of quasiconformal homeomorphisms:
\begin{lem}
\label{lem:composition}Let $\varphi: \Omega\to\widetilde{\Omega}$ be a $K$-quasiconformal homeomorphism. Then $\varphi$ generates by the composition rule $\varphi^{\ast}(f)=f\circ\varphi$ a bounded composition operator
$$
\varphi^{\ast}: L^1_p(\widetilde{\Omega})\to L^1_q(\Omega), \,\,\,1\leq q<p<\infty,
$$ 
if and only if, 
$$
Q_{p,q}(\Omega)=\left(\int\limits_{\Omega}|D\varphi|^{\frac{(p-n)q}{p-q}}~dx\right)^{\frac{p-q}{pq}}<\infty.
$$
\end{lem}

\begin{proof}
Let $f\in L^1_p(\widetilde{\Omega})$ be a smooth function. Then, because quasiconformal homeomorphisms possess the Luzin $N$ and $N^{-1}$ properties, the composition 
$g=\varphi^{\ast}(f)$ is well defined almost everywhere in $\Omega$ and belongs to $L^1_{1,\loc}(\Omega)$. Hence, using Theorem \ref{CompTh} we obtain
\begin{multline*}
\|g\mid L^1_{q}({\Omega})\|\leq
\biggl(\int\limits _{{\Omega}}\left(\frac{|D\varphi(x)|^p}{|J(x,\varphi)|}\right)^{\frac{q}{p-q}}~dx\biggr)^{\frac{p-q}{pq}}
\|f\mid L^1_{p}(\widetilde{\Omega})\|\\
=\biggl(\int\limits _{{\Omega}}\left(\frac{|D\varphi(x)|^n|D\varphi(x)|^{p-n}}{|J(x,\varphi)|}\right)^{\frac{q}{p-q}}~dx\biggr)^{\frac{p-q}{pq}}\|f\mid L^1_{p}(\widetilde{\Omega})\|\\
\leq K^{\frac{1}{p}} Q_{p,q}(\Omega)\|f\mid L^1_p(\widetilde{\Omega})\|.
\end{multline*}
By approximating an arbitrary function $f\in L^1_p(\widetilde{\Omega})$
by smooth functions we obtain the required inequality. 

Now, let the composition operator
$$
\varphi^*: L_p^1( \widetilde{\Omega}) \to L_q^1(\Omega),\,\,1\leq q<p<\infty,
$$
be bounded.
Then, using the  Hadamard inequality: 
$$
|J(x,\varphi)|\leq |D\varphi(x)|^n\,\,\text{for almost all}\,\, x\in\Omega,
$$ 
and Theorem~\ref{CompTh}, we have
$$
Q_{p,q}(\Omega)=\left(\int\limits_{\Omega}|D\varphi|^{\frac{(p-n)q}{p-q}}~dx\right)^{\frac{p-q}{pq}}\leq 
\left(\int\limits_\Omega \left(\frac{|D\varphi(x)|^{p}}{|J(x,\varphi)|}\right)^\frac{q}{p-q}~dx\right)^{\frac{p-q}{pq}}<+\infty. 
$$

\end{proof}

\begin{cor}
Let $\varphi: \Omega\to\widetilde{\Omega}$ be a $K$-quasiconformal homeomorphism such that 
$$
Q_{p,q}(\Omega)=\left(\int\limits_{\Omega}|D\varphi|^{\frac{(p-n)q}{p-q}}~dx\right)^{\frac{p-q}{pq}}<\infty.
$$
Then  
$$
\|\varphi^{\ast}f\mid L^1_q(\Omega)\|\leq K^{\frac{1}{p}}Q_{p,q}(\Omega)\|f\mid L^1_p(\widetilde{\Omega})\|\,\,\,\text{for any}\,\,\, f\in L^1_p(\widetilde{\Omega}).
$$
\end{cor}

\subsection{Sobolev-Poincar\'e inequalities}

Let $1\leq r,p\leq \infty$. 
We recall that
a bounded domain $\Omega$ in $\mathbb R^n$ is called a $(r,p)$-Sobolev-Poincar\'e domain, if for any function $f\in L^1_p(\Omega)$, the $(r,p)$-Sobolev-Poincar\'e inequality
$$
\inf\limits_{c\in\mathbb R}\|f-c\mid L_r(\Omega)\|\leq B_{r,p}(\Omega )\|\nabla f\mid L_p(\Omega)\|
$$
holds.

\vskip 0.5cm

{\bf Theorem B.}
{\it Let a bounded domain $\Omega$ in $\mathbb R^n$ be a $(r,q)$-Sobolev-Poincar\'e domain, $1<q\leq r<\infty$. Suppose that  there exists a $K$-quasiconformal homeomorphism $\varphi: \Omega\to\widetilde{\Omega}$ of a domain $\Omega$ onto a bounded domain $\widetilde{\Omega}$,
such that  $\varphi$   belongs to the Sobolev space $L^1_{\alpha}(\Omega)$ for some $\alpha>n$. Suppose additionally that
$$
Q_{p,q}(\Omega)=\left(\int\limits_{\Omega}|D\varphi|^{\frac{(p-n)q}{p-q}}~dx\right)^{\frac{p-q}{pq}}<\infty.
$$
for some $p\geq q$.
Then for $1\leq s=\frac{\alpha-n}{\alpha}r$ in the domain $\widetilde{\Omega}$ the $(s,p)$-Sobolev-Poincar\'e inequality
holds and 
$$
B_{s,p}(\widetilde{\Omega})\leq K^{\frac{1}{p}}
\min\limits_{1\leq q<p}\left( Q_{p,q}(\Omega)\|D\varphi\mid L_{\alpha}(\Omega)\|^{\frac{n}{s}}\right)
\cdot B_{r,q}(\Omega),
$$
where $B_{r,q}(\Omega)$ is the best constant in the $(r,q)$-Sobolev-Poincar\'e inequality in the domain $\Omega$.
}

\vskip 0.5cm

\begin{proof}
By the assumptions there exists a quasiconformal homeomorphism $\varphi: \Omega\to \widetilde{\Omega}$. Then, using the change of variable formula we obtain: 

$$
\inf\limits_{c\in \mathbb R}\biggl(\int\limits_{\widetilde{\Omega}}|f(y)-c|^s~dy\biggr)^{\frac{1}{s}}
=
\inf\limits_{c\in \mathbb R}\biggl(\int\limits_{\Omega}|f(\varphi(x))-c|^s|J(x,\varphi)|~dx\biggr)^{\frac{1}{s}}.
$$

Now we choose $r=\alpha s/(\alpha-n)$. Then, using the H\"older inequality we have:

\begin{multline}
\inf\limits_{c\in \mathbb R}\biggl(\int\limits_{\Omega}|f(\varphi(x))-c|^s|J(x,\varphi)|~dx\biggr)^{\frac{1}{s}}\\
\leq
\biggl(\int\limits_{\Omega}|J(x,\varphi)|^{\frac{r}{r-s}}~dx\biggr)^{\frac{r-s}{rs}}
\inf\limits_{c\in \mathbb R}\biggl(\int\limits_{\Omega}|g(x)-c|^r~dx\biggr)^{\frac{1}{r}}
\\
\leq
\biggl(\int\limits_{\Omega} |D\varphi(x)|^{\frac{nr}{r-s}}~dx\biggr)^{\frac{r-s}{rs}}
\inf\limits_{c\in \mathbb R}\biggl(\int\limits_{\Omega}|g(x)-c|^r~dx\biggr)^{\frac{1}{r}}\\
=
\biggl(\int\limits_{\Omega} |D\varphi(x)|^{\alpha}~dx\biggr)^{\frac{1}{\alpha}\frac{n}{s}}
\inf\limits_{c\in \mathbb R}\biggl(\int\limits_{\Omega}|g(x)-c|^r~dx\biggr)^{\frac{1}{r}}.
\nonumber
\end{multline}

Hence, applying the Sobolev-Poincar\'e inequality in the $(r,q)$-Sobolev-Poincar\'e domain $\Omega$
$$
\inf\limits_{c\in \mathbb R}\biggl(\int\limits_{\Omega}|g(x)-c|^r~dx\biggr)^{\frac{1}{r}}\leq B_{r,q}(\Omega)\biggl(\int\limits_{\Omega}|\nabla g(x|^q~dx\biggr)^{\frac{1}{q}}
$$
we have
$$
\inf\limits_{c\in \mathbb R}\biggl(\int\limits_{\Omega}|f(y)-c|^s~dy\biggr)^{\frac{1}{s}}\leq
\|D\varphi\mid L_{\alpha}(\Omega)\|^{\frac{n}{s}} B_{r,q}(\Omega) \|g\mid L^1_q(\Omega)\|.
$$
By Lemma \ref{lem:composition} we have
$$
\|g\mid L^1_q(\Omega)\| \leq K^{\frac{1}{p}}Q_{p,q}(\Omega)\|f\mid L^1_p(\tilde{\Omega})\|.
$$

Therefore
$$
\inf\limits_{c\in \mathbb R}\biggl(\int\limits_{\Omega}|f(y)-c|^s~dy\biggr)^{\frac{1}{s}}\leq K^{\frac{1}{p}}Q_{p,q}(\Omega)\|D\varphi\mid L_{\alpha}(\Omega)\|^{\frac{n}{s}}
 B_{r,q}(\Omega)\biggl(\int\limits_{\widetilde{\Omega}}|\nabla f|^p~dy\biggr)^{\frac{1}{p}}.
$$

\end{proof}

As an application we consider the Neumann spectral problem for the $p$-Laplace operator
 
\begin{gather}
\begin{cases}
-\dv(|\nabla u|^{p-2}\nabla u)=\mu_p |u|^{p-2}u\,\,\text{in}\,\,\Omega,\\
{\frac{\partial u}{\partial n}}\bigg|_{\partial\Omega}=0.
\end{cases}
\nonumber
\end{gather}

By the generalized version of Rellich-Kondrachov compactness theorem (see, for example, \cite{M}, \cite{HK}, \cite{EH-S}) and the $(r,p)$--Sobolev-Poincar\'e inequality for $r>p$ 
the embedding operator
$$
i: W^{1,p}(\Omega) \hookrightarrow L_p(\Omega)
$$
is compact in domains which satisfy conditions of Theorem B.

Hence, the first nontrivial Neumann eigenvalue $\mu_p(\Omega)$ can be characterized as
$$
\mu_p(\Omega)=\min\left\{\frac{\int\limits_{\Omega}|\nabla u(x)|^p~dx}{\int\limits_{\Omega}|u(x)|^p~dx}: u\in W^{1,p}(\Omega)\setminus\{0\},\,
\int\limits_{\Omega}|u|^{p-2}u~dx=0\right\}.
$$

Moreover, $\mu_p(\Omega)^{-\frac{1}{p}}$ is the best constant $B_{p,p}(\Omega)$ (for example we refer to \cite{BCT15}) in the following Poincar\'e inequality
$$
\inf\limits_{c\in\mathbb R}\|f-c \mid L^p(\Omega)\|\leq B_{p,p}(\Omega)\|\nabla f \mid L^p(\Omega)\|,\,\,\, f\in W^{1,p}(\Omega).
$$

So from Theorem B in the case $s=p$ we obtain

\begin{thm}
Let a bounded domain $\Omega$ in $\mathbb R^n$ be a $(r,q)$-Sobolev-Poincar\'e domain, $1<q\leq r<\infty$. Assume that there exists a $K$-quasiconformal homeomorphism $\varphi: \Omega\to\widetilde{\Omega}$ of a domain $\Omega$ onto a bounded domain $\widetilde{\Omega}$, so that $\varphi$ belongs to the
 space $L^1_{\alpha}(\Omega)$ for $\alpha=nr/(r-p)$,  $r>p$.  Suppose that
$$
Q_{p,q}(\Omega)=\left(\int\limits_{\Omega}|D\varphi|^{\frac{(p-n)q}{p-q}}~dx\right)^{\frac{p-q}{pq}}<\infty.
$$
for some $p>q$.
Then 
$$
\frac{1}{\mu_p(\widetilde{\Omega})}\leq K
\min\limits_{1\leq q<p}\left( Q^p_{p,q}(\Omega)\|D\varphi\mid L_{\alpha}(\Omega)\|^{n}\right)
\cdot B^p_{r,q}(\Omega),
$$
where $B_{r,q}(\Omega)$ is the best constant in the $(r,q)$-Sobolev-Poincar\'e inequality in the domain $\Omega$.
\end{thm}

In the limit case, when a quasiconformal mapping $\varphi:\Omega\to\widetilde{\Omega}$ is Lipschitz homeomorphism, we have: 

\vskip 0.5cm

{\bf Theorem A.} {\it  
Let a bounded domain $\Omega$ in $\mathbb R^n$ be a $(p,p)$-Sobolev-Poincar\'e domain, $1<p<\infty$, and there exists a Lipschitz $K$-quasiconformal homeomorphism $\varphi: \Omega\to\widetilde{\Omega}$ of a domain $\Omega$ onto a bounded domain $\widetilde{\Omega}$ such that 
$$
Q_{p}(\Omega)=\ess\sup\limits_{x\in\Omega}|D\varphi(x)|^{\frac{p-n}{p}}<\infty.
$$
Then 
$$
\frac{1}{\mu_p(\widetilde{\Omega})}\leq K
Q^p_{p}(\Omega)\||D\varphi|^n\mid L_{\infty}(\Omega)\|
\cdot \frac{1}{\mu_p(\Omega)}.
$$
}
\vskip 0.5cm

Let us give an illustration of Theorem A.

\vskip 0.5cm

{\bf Example C.}   
Consider the domain $\Omega_{\delta}=\Omega_1\cup\Omega_2$, $\delta>0$ given, 
$\alpha=\delta (\sqrt{3}-1)/2$, where 
\begin{equation*}
\Omega_1=\bigl\{(x',x_n)\in R^n\,: \max\{\vert x'\vert -\delta\,, -\alpha\} < x_n<\alpha\bigr\}
\end{equation*}
and
\begin{equation*}
\Omega_2=\bigl\{(x',x_n)\in R^n\,:  -\alpha < x_n<\min\{\delta -\vert x'\vert\,, \alpha\}\bigr\}\,.
\end{equation*}
Let $n=3$. 
Then, $\Omega_{\delta}=\Omega_1\cup\Omega_2$ is a $(\delta (\sqrt{3}-1)/2\,, \delta\sqrt{2}) $-John domain 
and there exists a $K$-quasiconformal mapping $\varphi:\mathbb R^3\to \mathbb R^3$ such that $\varphi(B^3(0,1))=\Omega_{\delta}$.

The domain 
$\Omega_{\delta}$ is starshaped with respect to the origin.
By \cite{GV}  with the angle $\alpha =\pi /12$ we obtain
\begin{equation*}
K^2\le2\frac{\sqrt{4+\sqrt{6}+\sqrt{2}}}{4-\sqrt{6}-\sqrt{2}} 
\end{equation*}
and
\begin{equation*}
\vert D\varphi (x)\vert^3\le
\delta^32^3\biggl(\frac{\sqrt{4+\sqrt{6}-\sqrt{2}}+\sqrt{4-\sqrt{6}+\sqrt{2}}}
{\sqrt{6}-\sqrt{2}}\biggr)^3\,.
\end{equation*}

By Theorem A for $p>3$
\begin{multline*}
\frac{1}{\mu_p(\Omega_{\delta})}\leq
\frac{\sqrt{2}(4+\sqrt{6}+\sqrt{2}))^{1/4}}
{\sqrt{4-\sqrt{6}-\sqrt{2}}}
\biggl(\frac{\sqrt{4+\sqrt{6}-\sqrt{2}}+\sqrt{4-\sqrt{6}+\sqrt{2}}}{(\sqrt{6}-\sqrt{2})
(\mu_p(B^3(0,1))^{1/p}
}2\delta\biggr)^p.
\end{multline*}
If $p=2$ then the first non-trivial Neumann eigenvalue in the unit ball  is
$$
\mu_2(B^n(0,1))=p_{n/2},
$$ 
where $p_{n/2}$ denotes the first positive zero of the function $(t^{1-n/2}J_{n/2}(t))'$. In particular, if $n=2$, we have $p_1=j_{1,1}'\approx1.84118$ where $j_{1,1}'$ denotes the first positive zero of the derivative of the Bessel function $J_1$.  And $p_{3/2}$ denotes the first positive zero of the function $(t^{1/2}J_{3/2}(t))'$.

If $p>2$, then by \cite{ENT}
$$
\mu_p(B^n(0,1)) \geq \left(\frac{\pi_p}{2}\right)^p
$$
where
$$
\pi_p = 2 \int\limits_0^{(p-1)^{\frac{1}{p}}} \frac{dt}{(1-t^p/(p-1))^{\frac{1}{p}}} =
2 \pi \frac{(p-1)^{\frac{1}{p}}}{p \sin(\pi/p)}.
$$

\section{Poincar\'e inequalities for Whitney complexes}

The aim of  this section is to study  the upper estimates of the Poincar\'e constants  $B_{p,p}(W)$ 
for the $(p,p)$-Poincar\'e-Sobolev inequalities: 
$$
\inf\limits_{c\in\mathbb R}\|f-c\mid L_p(W)\|\leq B_{p,p}(W)\|\nabla f\mid L_p(W)\|
$$
for functions $f$ of the space $L^{1}_{p}(W)$ defined in the fractal type domains what we call the \textit{Whitney complex 
$W$ in $\mathbb R^n$. }

\begin{lem}\label{lemma_1}\cite{H}
Let $1\le p <\infty$.
Let $A$ be a measurable subset of a domain $\Omega$ in $\mathbb R^n$ such that $\vert A\vert >0$ and let $f$ be in $L^p(\Omega)$. Then for each $c\in \mathbb R$
\begin{equation*}
\|f-f_A\mid L_p(\Omega)\|
\leq 
2\biggl(\frac{\vert \Omega\vert}{\vert A\vert}\biggr)^{1/p}
\|f-c\mid L_p(W)\|\,.
\end{equation*}
\end{lem}

\begin{lem}\label{lemma_2}{H}
Suppose that
\begin{equation*}
\|f-f_{Q_j}\mid L_p(Q_j)\|
\leq 
B_{p,p}(Q_j)
\|\nabla f\mid L_p(Q_j)\|\,,
\end{equation*}
where $B_{p,p}(Q_j)$ is the Poincar\'e constant of the domain $Q_j$, $j=1,2$,
and $\vert Q_1\cap Q_2\vert\neq \emptyset$. Then,
$$
\int_{Q_1\cup Q_2}
\vert f(x)-f_{Q_1\cup Q_2}\vert^p\,dx\leq \frac{2^{2p-1}}{\vert Q_1\cap Q_2\vert}
\sum_{j=1}^{2}\vert Q_j\vert B^p_{p,p}(Q_j)
\int_{Q_j}
\vert \nabla f(x)\vert^p\,dx\,.
$$

\end{lem}

\begin{proof}

By Lemma \ref{lemma_1}
\begin{multline*}
\int_{Q_1\cup Q_2}
\vert f(x)-f_{Q_1\cup Q_2}\vert^p\,dx
\le 2^p\int_{Q_1\cup Q_2}\vert f(x)-f_{Q_1\cap Q_2}\vert^p\,dx\\
\le
\frac{2^{2p-1}}{\vert Q_1\cap Q_2\vert}
\sum_{j=1}^{2}\vert Q_j\vert
\int_{Q_j}
\vert f(x)-f_{Q_j}\vert^p\,dx\\
\le
\frac{2^{2p-1}}{\vert Q_1\cap Q_2\vert}
\sum_{j=1}^{2}\vert Q_j\vert B^p_{p,p}(Q_j)
\int_{Q_j}
\vert \nabla f(x)\vert^p\,dx\,.
\end{multline*}

\end{proof}

Suppose that $Q_1$, $R_2$ and $Q_3$ are bounded convex domains.
If $\vert Q_1\cap R_2\vert >0$ and $\vert R_2\cap Q_3\vert >0$ 
and $Q_1$ and $Q_3$ are disjoint we call 
the set $A=Q_1\cup R_2 \cup Q_3$ 
a Whitney triple. 

Now we use  Lemma \ref{lemma_2} for Whitney triples.

\begin{lem}\label{lemma_3}
Let $A$ in $R^n$ be a Whitney triple.
Then,
$$
\int\limits_{A}\vert f(x)-f_{A}\vert^p\,dx\leq
B_{p,p}^p(A)
\int\limits_{A}\vert\nabla f(x)\vert^p\,dx\,,
$$
where 
\begin{multline*}
B_{p,p}^p(A)\leq 
2^{4p-1}\biggl(\frac{\vert Q_1\cup R_2\vert}{\vert R_2\vert }\frac{\vert Q_1\vert}{\vert Q_1\cap R_2\vert}\biggr)B_{p,p}^p(Q_1)\\
+2^{4p-1}\biggl(\frac{\vert Q_1\cup R_2\vert}{\vert Q_1\cap R_2\vert }+\frac{\vert Q_3\cup R_2\vert}{\vert Q_3\cap R_2\vert}\biggr)B^p_{p,p}(R_2)\\
+2^{4p-1}\biggl(\frac{\vert Q_3\cup R_2\vert}{\vert R_2\vert }\frac{\vert Q_3\vert}{\vert Q_3\cap R_2\vert}\biggr)B^p_{p,p}(Q_3)\,.
\end{multline*}
\end{lem}

\begin{proof}
Using  Lemma \ref{lemma_1}
\begin{multline*}
\|f-f_A\mid L_p(A)\|^p
\le 2^p
\|f-f_{R_2}\mid L_p(A)\|^p
\\
\le 2^p
\biggl(
\int_{Q_1\cup R_2}
\vert f(x)-f_{R_2}\vert^p\,dx+
\int_{R_2\cup Q_3}
\vert f(x)-f_{R_2}\vert^p\,dx\biggr)\\
\le 2^{2p}\frac{\vert Q_1\cup R_2\vert}{\vert R_2\vert}
\int_{Q_1\cup R_2}
\vert f(x)-f_{Q_1\cup R_2}\vert^p\,dx\\
+2^{2p}
\frac{\vert R_2\cup Q_3\vert}{\vert R_2\vert}
\int_{R_2\cup Q_3}
\vert f(x)-f_{R_2\cup Q_3}\vert^p\,dx\,.
\end{multline*}
By Lemma \ref{lemma_2}
\begin{multline*}
\|f-f_A\mid L_p(A)\|^p
\le \\ 2^{2p}
\frac{\vert Q_1\cup R_2\vert}{\vert R_2\vert }
\frac{2^{2p-1}}{\vert Q_1\cap R_2\vert}
\left(\vert Q_1\vert B^p_{p,p}(Q_1)\int_{Q_1}\vert\nabla f(y)\vert^p\,dy+\vert R_2\vert B^p_{p,p}(R_2)\int_{R_2}\vert\nabla f(y)\vert^p\,dy\right)\\
+
2^{2p}
\frac{\vert R_2\cup Q_3\vert}{\vert R_2\vert }
\frac{2^{2p-1}}{\vert R_2\cap Q_3\vert}
\left(\vert R_2\vert B^p_{p,p}(R_2)\int_{R_2}\vert\nabla f(y)\vert^p+\vert Q_3\vert B^p_{p,p}(Q_3)\int_{Q_3}\vert\nabla f(y)\vert^p\,dy\,dy\right)\,.
\end{multline*}
\end{proof}

\begin{defn}
If $A_j$ are Whitney triples and $\vert A_j\cap A_{j+1}\vert >0$ we call
the set $W=\bigcup\limits_{j=1}^{\infty}A_j$ 
a  Whitney complex.
\end{defn}

\begin{thm}\label{theorem_1}
Let 
\begin{equation*}
W=\bigcup_{i=1}^{\infty}A_j\,,
\end{equation*}
be a Whitney complex.
Then
\begin{multline*}
\int\limits_W\vert f(x)-f_{A_1}\vert^p\,dx
\le 2^{p-1}\sum_{j=1}^{\infty}B^p_{p,p}(A_j)\int_{A_j}\vert\nabla f(x)\vert ^p\,dx\\
+
2^{2p}\sum_{j=1}^{\infty}
\int_{A_j}
j^{p-1}
\sum_{\mu =1}^{j}
\frac{B^p_{p,p}(A_{\mu})}{\vert A_{\mu}\cap A_{\mu+1}\vert }
\int_{A_{\mu}}\vert\nabla f(x)\vert^p\,dx\,dy\,.
\end{multline*}
\end{thm}

\begin{proof}
We have to estimate the integral
\begin{multline*}
\int_W\vert f(x)-f_{A_1}\vert^p\,dx
\le 2^{p-1}\sum_{j=1}^{\infty}\int_{A_j}\vert f(x)-f_{A_j}\vert^p\,dx
+2^{p-1}\sum_{j=1}^{\infty}\int_{A_j}
\vert f_{A_j}-f_{A_1}\vert^p\,dx\,.
\end{multline*}
The integral
\begin{equation*}
\int_{A_j}\vert f(x)-f_{A_j}\vert^p\,dx 
\end{equation*}
was handled
in Lemma \ref{lemma_3}.
We estimate the integral
\begin{equation*}
\int_{A_i}\vert f_{A_i}-f_{A_1}\vert^p\,dx\,.
\end{equation*}
Let us write
$
f_{A_j}=f_j\,.
$
By the triangle inequality
\begin{align*}
\vert f_i-f_1\vert\le\sum_{k =1}^{i-1}
\vert f_{k}-f_{k +1}\vert
\end{align*}
where
\begin{equation*}
\vert f_{k}-f_{k +1}\vert^p
\le
\frac{2^{p-1}}{\vert A_{k}\cap A_{k +1}\vert}
\biggl(
\int_{A_{k}}\vert f(x)-f_{k}\vert^p\,dx+
\int_{A_{k +1}}\vert f(x)-f_{k +1}\vert^p\,dx\biggr)\,.
\end{equation*}
Hence,
\begin{multline*}
\int_{A_i}\vert f_{A_i}-f_{A_1}\vert^p\,dx\le \int_{A_i}\biggl(\sum_{k =1}^{i-1}\vert f_{k}-f_{k +1}\vert\biggr) ^p\,dx\\
\le 
\int_{A_i}
\biggl(
\sum_{k =1}^{i-1}
\frac{2^{1-1/p}}{\vert A_{k}\cap A_{k +1}\vert ^{1/p}}
\biggl(
\int_{A_{k}}\vert f(x)-f_{k}\vert^p\,dx+
\int_{A_{k +1}}\vert f(x)-f_{k +1}\vert^p\,dx
\biggr)^{1/p}
\biggr)^{p}\,.
\end{multline*}
Thus, by Lemma \ref{lemma_3}
\begin{multline*}
\int_{A_i}\vert f_{A_i}-f_{A_1}\vert^p\,dy
\le 
\int_{A_i}\biggl(\sum_{k =1}^{i}
\biggl(
\frac{2^{p}}{\vert A_{k}\cap A_{k +1}\vert}
\int_{A_{k}}\vert f(x)-f_{k}\vert^p\,dx\biggr)^{1/p}\biggr)^{p}\,dy\\
\le 
\int_{A_i}\biggl(\sum_{k =1}^{i}
\biggl(
\frac{2^{p}B^p_{p,p}(A_{k})}{\vert A_{k}\cap A_{k+1}\vert}
\int_{A_{k}}\vert \nabla f(x)\vert^p\,dx\biggr)^{1/p}\biggr)^{p}\,dy\\
\le
\int_{A_i}
i^{p-1}
\sum_{k =1}^{i}2^p
\frac{B^p_{p,p}(A_{k})}{\vert A_{k}\cap A_{k+1}\vert }
\int_{A_{k}}\vert\nabla f(x)\vert^p\,dx\,dy\,.
\end{multline*}

Hence,
\begin{multline*}
\int_W\vert f(x)-f_{A_1}\vert^p\,dx
\le 2^{p-1}\sum_{j=1}^{\infty}B^p_{p,p}(A_{j})\int_{A_j}\vert\nabla f(x)\vert ^p\,dx\\
+
2^{2p}\sum_{j=1}^{\infty}
\int_{A_j}
j^{p-1}
\sum_{k =1}^{j}
\frac{B^p_{p,p}(A_{k})}{\vert A_{k}\cap A_{k+1}\vert }
\int_{A_{k}}\vert\nabla f(x)\vert^p\,dx\,dy\,.
\end{multline*}
\end{proof}

\begin{thm}\label{theorem_2}
Let 
\begin{equation*}
W=\bigcup_{i=1}^{\infty}A_j\,,
\end{equation*}
be a Whitney complex.
Then
\begin{multline*}
\int_W\vert f(x)-f_{A_1}\vert^p\,dx
\le 2^{p-1}\sum_{j=1}^{\infty}B^p_{p,p}(A_j)\int_{A_j}\vert\nabla f(x)\vert ^p\,dx\\
+
2^{2p}\sum_{i=1}^{\infty}
\sum_{k =i}^{\infty}
{k}^{p-1}
\vert A_{k}\vert
\frac{B^p_{p,p}(A_i)}{\vert A_{i}\cap A_{i+1}\vert }
\int_{A_{i}}\vert\nabla f(x)\vert^p\,dx\,dy\,.
\end{multline*}
\end{thm}

\begin{proof}
There is a reformulation of the second term on the right hand side in Theorem \ref{theorem_1}.
\end{proof}

The following extension theorem is needed for fractal type examples.

\begin{thm}\label{theorem_3}
Let $W$ be a fractal tree in $\mathbb R^n$, $n\geq 2$, and let $\Delta_0$ be the starting domain for the tree $W$.
Let 
\begin{equation*}
W=\cup \Delta _k 
\end{equation*}
where $\Delta_k$ are the elements of the tree.
When the element $\Delta_k$ is fixed, let
\begin{align*}
&(\Delta ,\Delta_0, \Delta_k)=\\
&\biggl\{\mbox{ all the elements $\Delta $ of the tree to where we go from $\Delta_0$ through
$\Delta_k$ }\biggr\}\,.
\end{align*}
Let $1\le p <\infty$.
Let $B_{p,p}(\Delta _k)$ be the Poincar\'e constant of the element of $\Delta _k$ of the tree.
Then,
\begin{multline*}
\int_{W}\vert f(x)-f_{\Delta_0}\vert^p\,dx\le
2^{p-1}
\sum_{\Delta_k}
B_{p,p}^p(\Delta_k)\int_{\Delta_k}\vert\nabla f(x)\vert^p\,dx+\\
2^{p-1}
\sum_{ all  \Delta_k}
\sum_{\Delta\in (\Delta_0,\Delta_k,\Delta )}
\#\{ \mbox{ steps from }\Delta_0\mbox{  to }\Delta\} ^{p-1}
\vert \Delta\vert \frac{B_{p,p}^p(\Delta_k)}{\vert \Delta_k\vert}\int_{\Delta_k}
\vert\nabla f(x)\vert^p\,dx\,.
\end{multline*}
\end{thm}

\begin{proof}
A modification of  the proof of Theorem 4.4 in \cite{H}.
\end{proof}

This theorem allows estimate the Poincar\'e constants in fractal domains:

\begin{exa}
Let $W$ be a fractal tree which is a modification of
the snowflake definition given in \cite{R}.

\begin{enumerate}
\item The starting point is a triangle $\Delta _0$  with the sidelength $a$.\\
\item We form three new triangles  $\Delta_1$  with the sidelength $a/3$, it is the 1st step.\\
\item We form for the previous  three triangles $\Delta_1$  each two new tringles $\Delta_2$
with the sidelength $a/3^2$, so altogether
$3\times 2$ new triangles $\Delta_2$, it is the 2nd step.\\
\item We form for the previous $3\times 2^{j-1}$ triangles 
$\Delta_{j}$, to
  each $\Delta_{j}$ two new triangles 
with the sidelength $a/3^{j+1}$ so altogether
$3\times 2^{j}$ new triangles  $\Delta_{j+1}$, it is the $j+1$ step.\\
\end{enumerate}

So, when we have $\Delta_j$ from the step $j$, its sidelength is $a/3^j$ and its area
is
\begin{equation*}
\frac{\sqrt{3}a^2}{4\cdot 3^{2j}}
\end{equation*}
and
\begin{equation*}
\#\biggl\{\Delta _j:\vert \Delta _j\vert=\frac{\sqrt{3}a^2}{2^2\cdot 3^{2j}}\biggr\} =3\cdot 2^{j-1}\,.
\end{equation*}

We denote by $\Delta_j^*$ the triangle which is obtained from $\Delta_j$ by extending
$\Delta_j$ to inside $\Delta_{j-1}$ so that
\begin{equation*}
c_1\vert\Delta_j\vert\le \vert \Delta_{j-1}\cap \Delta_j^*\vert \le c_2 \vert \Delta_j\vert\,.
\end{equation*}

We have to estimate the last term in Theorem \ref{theorem_3}:
\begin{multline*}
\sum_{j=1}^{\infty}
\#\biggl\{\Delta _j:\vert \Delta _j\vert=\frac{\sqrt{3}a^2}{2^2\cdot 3^{2j}}\biggr\}
\sum_{\Delta_i\in (\Delta_0,\Delta_j,\Delta_i)}
i^{p-1}2^{i-j}\vert \Delta_i^*\vert
\frac{B_{p,p}^p(\Delta_j^*)}{\vert \Delta_j^*\vert}
\int_{\Delta _j^*}\vert \nabla u(x)\vert^p\,dx\\
\le\sum_{j=1}^{\infty}
3\cdot2^{j-1}\sum_{i=j}^{\infty}i^{p-1}\frac{2^{i-j}}{3^{2i}}\frac{2^p}{\pi_p^p}
\frac{a^p}{3^{j(p-2)}}\int_{\Delta _j^*}\vert \nabla u(x)\vert^p\,dx\\
\le\sum_{j=1}^{\infty}
3\cdot2^{j-1}2^{-j}\sum_{i=j}^{\infty}i^{p-1}\frac{2^{i}}{3^{2i}}\frac{2^p}{\pi_p^p}
\frac{a^p}{3^{j(p-2)}}\int_{\Delta _j^*}\vert \nabla u(x)\vert^p\,dx\,.
\end{multline*}
Since,
\begin{equation*}
\sum_{i=j}^{\infty}
i^{p-1}\biggl(\frac{2}{3^2}\biggr)^{i}
\le 
\biggl(\frac{2}{3^2}\biggr)^{j(1-\epsilon )}\,,
\end{equation*}
where $\epsilon >0$ is an arbitrary small number,
\begin{multline*}
\sum_{j=1}^{\infty}
3\cdot2^{j-1}2^{-j}\sum_{i=j}^{\infty}i^{p-1}\frac{2^{i}}{3^{2i}}\frac{2^p}{\pi_p^p}
\frac{a^p}{3^{j(p-2}}\\
\le 3\frac{2^{p-1}}{\pi _p^p}a^p\sum_{j=1}^{\infty}
\biggl(\frac{2}{3^2}\biggr)^{j(1-\epsilon)}\frac{1}{3^{j(p-2)}}=
3\frac{2^{p-1}}{\pi _p^p}a^p\sum_{j=1}^{\infty}
\frac{2^{j(1-\epsilon)}}{3^{j(p-2\epsilon)}}\,,
\end{multline*}
where the sum converges, since $\epsilon >0$ can be taken arbitrarily small.
\end{exa}

Results of Section~2 allow obtain variation of Poincar\'e constants under quasiconformal perturbations fractal type domains.

\begin{thm}
Let $\widetilde{W}$ be an image of the Whitney complex $W$ under a Lipschitz $K$-quasiconformal homeomorphism $\varphi: W\to\widetilde{W}$ such that
$$
Q_{p}(W)=\ess\sup\limits_{x\in W}|D\varphi(x)|^{\frac{p-n}{p}}<\infty.
$$
Then 
$$
\frac{1}{\mu_p(\widetilde{W})}\leq K
Q^p_{p}(W)\||D\varphi|^n\mid L_{\infty}(W)\|
\cdot \frac{1}{\mu_p(W)}.
$$
 
\end{thm}

\end{document}